\documentclass[10pt]{article}
\usepackage{graphicx}
\usepackage{amssymb,latexsym,amscd,amsmath,amsfonts,amsthm,pb-diagram}
\usepackage[mathscr]{eucal}
\usepackage{floatflt}
\numberwithin{equation}{section}
\newtheorem{theorem}{Theorem}[section]
\newtheorem{proposition}[theorem]{Proposition}
\newtheorem{corollary}[theorem]{Corollary}
\newtheorem{lemma}[theorem]{Lemma}

\begin{document}
\begin{center} {\Large A Splitting Theorem for Local Cohomology \\
 and its Applications}
 \footnote [1]{
\indent{\bf AMS Classification 2010:} 13D45; 13H10.\\
This work is supported in part by NAFOSTED (Vietnam).}
\end{center}
\begin{center}
                \textsc{Nguyen Tu Cuong and Pham Hung Quy}\\
\end{center}
\begin{abstract}
Let $R$ be a commutative Noetherian ring and $M$  a finitely
generated $R$-module. We show in this paper that, for  an integer
$t$,  if
 the local cohomology module $H^{i}_\mathfrak{a}(M)$ with respect to an ideal
 $\frak a$ is finitely generated for all $i<t$, then
 $$H^{i}_\mathfrak{a}(M/xM)\cong H^{i}_\mathfrak{a}(M)\oplus H^{i+1}_\mathfrak{a}(M)$$
 for all $\frak a$-filter
regular  elements $x$ containing in a enough large power of $\frak
a$ and all $i<t-1$. As consequences we obtain generalizations, by
very short proofs, of the main results of M. Brodmann and A.L.
Faghani (A finiteness result for associated primes of local
cohomology modules,
 Proc. Amer. Math. Soc., 128(2000), 2851-2853)
 and of H.L. Truong and the first author (Asymptotic behavior of parameter ideals in generalized Cohen-Macaulay
module,  J. Algebra, 320(2008),158-168).
\end{abstract}
\section{Introduction}

A finitely generated module $M$  of dimension $d>0$ over a
Noetherian local ring $(R, \mathfrak{m})$ is called a generalized
Cohen-Macaulay module (see \cite{CST}), if there exists a positive
integer $k$ such that $\mathfrak{m}^{k}H^{i}_\mathfrak{m}(M) = 0$
for all $i <d$, where $H^{i}_\mathfrak{m}(M)$ is the $i$-th local
cohomology module of $M$ with respect to the maximal ideal $\frak
m$. Then the following split property of local cohomology modules
is useful in the theory of generalized Cohen-Macaulay modules  (see \cite{S}): For
a parameter element $x$ of $M$ there exists a enough large integer
$n$ such that
 $H^{i}_\mathfrak{m}(M/x^nM)\cong H^{i}_\mathfrak{m}(M)\oplus
H^{i+1}_\mathfrak{m}(M)$ for all $i<d-1$. It should be noted here
that this integer $n$ is in general dependent on the choice of the
element
$x$. It raises to the following natural question.\\
{\bf Question.} Let $M$ be a generalized Cohen-Macaulay module.
Does there exist a positive integer $n$
 such that  for any parameter element $x $ of $M$ containing in $\mathfrak{m}^n$,
 it holds $H^{i}_\mathfrak{m}(M/xM)\cong H^{i}_\mathfrak{m}(M)\oplus
H^{i+1}_\mathfrak{m}(M)$ for all $i<d-1$?

The purpose of this paper is not only to find an answer to this
question but also to prove a more general split property of local
cohomology modules as follows. Let $R$ be a Noetherian ring ($R$
is not necessary to be a local ring) and $\frak a$ an ideal of
$R$. An element $x\in \frak a$ is called an $\frak a$-filter
regular element of $M$ if $x \notin \mathfrak{p}$ for all
$\mathfrak{p} \in \mathrm{Ass}M\setminus V(\mathfrak{a})$, where
$V(\mathfrak{a})$ is the set of all prime ideals of $R$ containing
$\mathfrak{a}$.

\begin{theorem}\label{th}
Let $M$ be a finitely generated module over a Noetherian ring $R$
and  $\mathfrak{a}$ an ideal of  $R$. Let $t, n_0$ be positive
integers such that $\mathfrak{a}^{n_0}H^{i}_\mathfrak{a}(M) = 0$
for all $i<t$. Then, for all  $\mathfrak{a}$-filter regular
element $x \in \mathfrak{a}^{2n_0}$ of $M$, it holds
$$H^{i}_\mathfrak{a}(M/xM)\cong H^{i}_\mathfrak{a}(M)\oplus H^{i+1}_\mathfrak{a}(M),$$
 for all $i<t-1$, and
$$0:_{H^{t-1}_\mathfrak{a}(M/xM)}\mathfrak{a}^{n_0} \cong
H^{t-1}_\mathfrak{a}(M)\oplus
0:_{H^{t}_\mathfrak{a}(M)}\mathfrak{a}^{n_0}.$$
\end{theorem}

It is well-known that every parameter element of $M$ is an $\frak
m$-filter regular element, if $M$ is a generalized Cohen-Macaulay
module. Therefore Theorem \ref{th} gives a complete affirmative
answer for the question above, where the integer $n$ is just
$n=\min \{ k\mid  \mathfrak{m}^{k}H^{i}_\mathfrak{m}(M) = 0,\
i=0,\ldots , d-1\}$. The key point for proving Theorem \ref{th} is
as follows. Let $x$ and $ t$ be as in Theorem \ref{th}. From the
short exact sequence $0 \longrightarrow M/H^{0}_\mathfrak{a}(M)
\overset{x}{\longrightarrow}M \longrightarrow M/xM \longrightarrow
0$ we obtain short exact sequences
$$0 \longrightarrow H^{i}_\mathfrak{a}(M) \longrightarrow H^{i}_\mathfrak{a}(M/xM) \longrightarrow
H^{i+1}_\mathfrak{a}(M) \longrightarrow 0, \quad (\ast)$$
$i=0,\ldots ,t-2$. So for each $i<t-1$ we can consider the short
exact sequence $(\ast)$ as an element of the group of extensions $
\mathrm{Ext}^1_R(H^{i+1}_\mathfrak{a}(M), H^{i}_\mathfrak{a}(M))$
(see, Chapter 3, \cite{M}). Then, the splitting of sequence
$(\ast)$ is equivalent to say that it is the zero-element of this
group. We will give some properties of the (Bear) sum and the
$R$-module structure of this group of extensions in the next
section. The proof of Theorem \ref{th} will be done in Section 3.
The last section is involved to find applications of Theorem
\ref{th}. Especially, we show that the main theorems of M.
Brodmann and A.L. Faghani \cite{BF}, and of H.L. Truong and the
first author \cite{CT} are immediate consequences of Theorem
\ref{th}.

\section{The extension module Ext$^1$}
  In this section, let $\mathfrak{a}$ be an ideal of
a Noetherian ring $R$, and $M$ a finitely generated $R$-module.
It is well-known that for a positive integer $t$, $H^{i}_\mathfrak{a}(M)$ is finitely
generated for all $i<t$ iff there exists a positive
integer $n_0$ such that $\mathfrak{a}^{n_0}
H^{i}_\mathfrak{a}(M) = 0$ for all $i<t$.  An element $x\in \frak a$ is called an $\frak a$-filter
regular element of $M$ if $x \notin \mathfrak{p}$ for all
$\mathfrak{p} \in \mathrm{Ass}M\setminus V(\mathfrak{a})$, where
$V(\mathfrak{a})$ is the set of all prime ideals of $R$ containing
$\mathfrak{a}$. It should be noted that  there  always  exist $\frak a$-filter
regular elements. Moreover,  if $x \in \mathfrak{a}^{n_0}$ is an
$\mathfrak{a}$-filter regular element of $M$, then the short exact
sequence
$$0 \longrightarrow M' \overset{x}{\longrightarrow} M \longrightarrow M/xM \longrightarrow 0,$$
where ${M'} = M/H^{0}_\mathfrak{a}(M)$, reduces short
exact sequences
$$0 \longrightarrow H^{i}_\mathfrak{a}(M) \longrightarrow H^{i}_\mathfrak{a}(M/xM)
\longrightarrow H^{i+1}_\mathfrak{a}({M'}) \longrightarrow 0,$$
for all $i<t-1$. This situation is a special case of the following
consideration: given an integer $t$, an ideal $\frak a$ of $R$ and
a submodule $U$ of $M$. Set $\overline{M} =M/U$. We say that an
element $x \in \mathfrak{a}$ satisfies the condition $(\sharp)$
if $0:_Mx = U$, and the short exact sequence
$$0 \longrightarrow \overline{M} \overset{x}{\longrightarrow} M \longrightarrow M/xM \longrightarrow 0$$
reduces  short exact sequences
$$0 \longrightarrow H^{i}_\mathfrak{a}(M) \longrightarrow H^{i}_\mathfrak{a}(M/xM)
\longrightarrow H^{i+1}_\mathfrak{a}(\overline{M}) \longrightarrow
0$$ for all $i<t-1$.

\begin{proposition}\label{PRO3.1}
Let $M, U, \overline{M}, \mathfrak{a}$ and $t$ be as above.
Suppose that $x, y $ are elements in $\frak a$  such that $x$ and
$xy$ satisfy the condition $(\sharp)$, and $yH^{i}_\mathfrak{a}(M)
= 0$ for all $i<t$. Then, for all $i < t-1$, we have
$$H^{i}_\mathfrak{a}(M/xyM)\cong H^{i}_\mathfrak{a}(M)\oplus H^{i+1}_\mathfrak{a}(\overline{M}).$$
Moreover, if $H^{t}_\mathfrak{a}(\overline{M}) \cong
H^{t}_\mathfrak{a}(M)$,  we  have
$$0:_{H^{t-1}_\mathfrak{a}(M/xyM)}x \cong  0:_{H^{t-1}_\mathfrak{a}(M)}x \oplus 0:_{H^{t}_\mathfrak{a}(\overline{M})}x .$$
\end{proposition}

\begin{proof}
Since $U = 0:_Mx = 0:_Mxy$, we have the following commutative
diagram
\[\divide\dgARROWLENGTH by 2
\begin{diagram}
\node{0}\arrow{e}\node{\overline{M}}
\arrow{e,t}{x}\arrow{s,l}{id}\node{M}\arrow{e,t}{p_1}\arrow{s,l}{y}
\node{M/xM}\arrow{s,l}{f}\arrow{e} \node{0}\\
\node{0}\arrow{e}\node{\overline{M}}
\arrow{e,t}{xy}\node{M}\arrow{e,t}{p_2}\node{M/xyM}\arrow{e}
\node{0}
\end{diagram}
\]
with  exact rows, $p_1, p_2$ are natural projections, and $f$ is the
induced homomorphism.  We get by applying the functor
$H^{i}_\mathfrak{a}(\bullet)$ to the above diagram for all
$i < t-1$ the following commutative diagram
\[\divide\dgARROWLENGTH by 2
\begin{diagram}
\node{0}\arrow{e}\node{H^{i}_\mathfrak{a}(M)}
\arrow{e,t}{H^{i}_\mathfrak{a}(p_1)}\arrow{s,l}{y}\node{H^{i}_\mathfrak{a}(M/xM)}\arrow{e,t}{\delta_1^i}\arrow{s,l}
{H^{i}_\mathfrak{a}(f)}
\node{H^{i+1}_\mathfrak{a}(\overline{M})}\arrow{s,l}{id}\arrow{e} \node{0}\\
\node{0}\arrow{e}\node{H^{i}_\mathfrak{a}(M)}
\arrow{e,t}{H^{i}_\mathfrak{a}(p_2)}\node{H^{i}_\mathfrak{a}(M/xyM)}\arrow{e,t}{\delta_2^i}
\node{H^{i+1}_\mathfrak{a}(\overline{M})}\arrow{e} \node{0,}
\end{diagram}
\]
 where $\delta_1^i, \delta_2^i$ are connected homomorphisms. Moreover,
 since $yH^{i}_\mathfrak{a}(M) = 0$ for all $i<t-1$, $H^{i}_\mathfrak{a}(f)\circ H^{i}_\mathfrak{a}(p_1) =
 0$. Therefore there exists a homomorphism
 $$\epsilon^i:
 H^{i+1}_\mathfrak{a}(\overline{M}) \cong \mathrm{coker}H^{i}_\mathfrak{a}(p_1)\rightarrow
 H^{i}_\mathfrak{a}(M/xyM)$$
for all $i<t-1$, which makes the following diagram
 \[\divide\dgARROWLENGTH by 2
\begin{diagram}
\node{0}\arrow{e}\node{H^{i}_\mathfrak{a}(M)}
\arrow{e,t}{H^{i}_\mathfrak{a}(p_1)}\arrow{s,l}{y}\node{H^{i}_\mathfrak{a}(M/xM)}\arrow{e,t}{\delta_1^i}\arrow{s,l}
{H^{i}_\mathfrak{a}(f)}
\node{H^{i+1}_\mathfrak{a}(\overline{M})}\arrow{s,l}{id}\arrow{sw,t}{\epsilon^i}\arrow{e} \node{0}\\
\node{0}\arrow{e}\node{H^{i}_\mathfrak{a}(M)}
\arrow{e,t}{H^{i}_\mathfrak{a}(p_2)}\node{H^{i}_\mathfrak{a}(M/xyM)}\arrow{e,t}{\delta_2^i}
\node{H^{i+1}_\mathfrak{a}(\overline{M})}\arrow{e} \node{0}
\end{diagram}
\]
commutative for all $i<t-1$. Hence $\delta_2^i \circ \epsilon^i =
id$,  and so  we get
$$H^{i}_\mathfrak{a}(M/xyM)\cong H^{i}_\mathfrak{a}(M)\oplus H^{i+1}_\mathfrak{a}(\overline{M}),$$
for all $i < t-1$.\\
In the case $H^{t}_\mathfrak{a}(\overline{M}) \cong
H^{t}_\mathfrak{a}(M)$ and $i=t-1$, we have the following commutative diagram
\[\divide\dgARROWLENGTH by 2
\begin{diagram}
\node{0}\arrow{e}\node{H^{t-1}_\mathfrak{a}(M)}
\arrow{e,t}{H^{t-1}_\mathfrak{a}(p_1)}\arrow{s,l}{y}\node{H^{t-1}_\mathfrak{a}(M/xM)}\arrow{e,t}{\delta_1^{t-1}}\arrow{s,l}
{H^{t-1}_\mathfrak{a}(f)}
\node{0:_{H^{t}_\mathfrak{a}(\overline{M})}x}\arrow{s,l}{\alpha}\arrow{e} \node{0}\\
\node{0}\arrow{e}\node{H^{t-1}_\mathfrak{a}(M)}
\arrow{e,t}{H^{t-1}_\mathfrak{a}(p_2)}\node{H^{t-1}_\mathfrak{a}(M/xyM)}\arrow{e,t}{\delta_2^{d-1}}
\node{0:_{H^{t}_\mathfrak{a}(\overline{M})}xy}\arrow{e} \node{0,}
\end{diagram}
\]
where $\alpha$: $0:_{H^{t}_\mathfrak{a}(\overline{M})}x
\rightarrow 0:_{H^{t}_\mathfrak{a}(\overline{M})}xy$ is injective.
With similar method as used in the cases $i<t-1$, there exists a homomorphism
$\epsilon^{t-1}$:
$0:_{H^{t}_\mathfrak{a}(\overline{M})}x\rightarrow
H^{t-1}_\mathfrak{a}(M/xyM)$ such that $\delta_2^{t-1} \circ
\epsilon^{t-1} = \alpha$.
By applying  the functor
$\mathrm{Hom}_R(R/(x),\bullet)$ to the above diagram we can check that
$$0:_{H^{t-1}_\mathfrak{a}(M/xyM)}x \cong  0:_{H^{t-1}_\mathfrak{a}(M)}x \oplus 0:_{H^{t}_\mathfrak{a}(\overline{M})}x.$$
\end{proof}

If $x \in \mathfrak{a}$ satisfies the condition $(\sharp)$, for each
$i<t-1$ we can consider
$$0 \longrightarrow H^{i}_\mathfrak{a}(M) \longrightarrow H^{i}_\mathfrak{a}(M/xM)
 \longrightarrow H^{i+1}_\mathfrak{a}(\overline{M}) \longrightarrow 0$$
 as an extension of $H^{i}_\mathfrak{a}(M)$ by $H^{i+1}_\mathfrak{a}(\overline{M})$, therefore as an
 element of $\mathrm{Ext}^1_R(H^{i+1}_\mathfrak{a}(\overline{M}),
H^{i}_\mathfrak{a}(M))$ (see, Chapter 3, \cite{M}). We denote this element
by $E_x^i$. Especially, if $H^{t}_\mathfrak{a}(\overline{M}) \cong
H^{t}_\mathfrak{a}(M)$, we have the  short
exact sequence
$$0 \longrightarrow H^{t-1}_\mathfrak{a}(M) \longrightarrow H^{t-1}_\mathfrak{a}(M/xM)
 \longrightarrow 0:_{H^{t}_\mathfrak{a}(\overline{M})}x \longrightarrow 0.$$
Let $n_0$ be a positive integer such that $x \in \frak{a}^{n_0}$.
 Suppose that the short exact sequence above derives  the following short exact sequence
$$0 \longrightarrow 0:_{H^{t-1}_\mathfrak{a}(M)}\mathfrak{a}^{n_0} \longrightarrow
0:_{H^{t-1}_\mathfrak{a}(M/xM)}\mathfrak{a}^{n_0}
 \longrightarrow 0:_{H^{t}_\mathfrak{a}(\overline{M})}\mathfrak{a}^{n_0} \longrightarrow 0.$$
 Then we can consider this exact sequence as an element of  $\mathrm{Ext}^1_R(0:_{H^{t}_\mathfrak{a}(\overline{M})}\mathfrak{a}^{n_0},
0:_{H^{t-1}_\mathfrak{a}(M)}\mathfrak{a}^{n_0})$, and denote it by $F^{t-1}_{n_0,x}$.
 It should be noted here that an extension of $R$-module $A$ by $R$-module $C$ is split if it is the zero-element of
 $\mathrm{Ext}^1_R(C, A)$. The following result is important for the proof of Theorem 1.1.

\begin{theorem}\label{T3.2}
Let  $M, U, \overline{M}, \mathfrak{a}$ and $t$ be as above and
$x, y  \in\frak a$. Then the following statements are true.

\noindent
 $\mathrm{(i)}$ Suppose that $x, y,x+y$ satisfy the condition $(\sharp)$, then $E_{x+y}^i = E_x^i + E_y^i$ for all
$i<t-1$. Furthermore, if $H^{t}_\mathfrak{a}(\overline{M}) \cong
H^{t}_\mathfrak{a}(M)$ and $F^{t-1}_{n_0,x}, F^{t-1}_{n_0,y}$ are
determined, then $F^{t-1}_{n_0,x+y}$ is
also determined, and we have $F^{t-1}_{n_0,x+y} = F^{t-1}_{n_0,x} + F^{t-1}_{n_0,y}$.

\noindent
$\mathrm{(ii)}$ Suppose that  $x, xy$ satisfy the condition $(\sharp)$, then $E_{xy}^i =
yE_x^i$ for all $i<t-1$. Moreover, if
$H^{t}_\mathfrak{a}(\overline{M}) \cong H^{t}_\mathfrak{a}(M)$ and
$F^{t-1}_{n_0,x}$ is determined, then $F^{t-1}_{n_0,xy}$ is also
determined and $F^{t-1}_{n_0,xy}=yF^{t-1}_{n_0,x}$. Especially,  if
$yH^{i}_\mathfrak{a}(M)=0$, for all $i<t$, then $F^{t-1}_{n_0,xy}=E_{xy}^i =0$ for
all $i<t-1$.
\end{theorem}

\begin{proof} (i) We consider the homomorphism $\varphi: M \to M \oplus M,\, \varphi(m) = (xm,
ym)$. Because $U = 0:_Mx=0:_My$ so we have short exact sequence
$$0 \longrightarrow
\overline{M} \overset{\overline{\varphi}}{\longrightarrow} M
\oplus M \longrightarrow N \longrightarrow 0,$$ where $N =
\mathrm{coker}(\overline{\varphi})$. The following diagram is
commutative
\[\divide\dgARROWLENGTH by 2
\begin{diagram}
\node{0}\arrow{e}\node{\overline{M}}
\arrow{e,t}{\overline{\varphi}}\arrow{s,l}{\Delta_{\overline{M}}}\node{M
\oplus M}\arrow{e}\arrow{s,l}{id}
\node{N}\arrow{e}\arrow{s} \node{0}\\
\node{0}\arrow{e}\node{\overline{M} \oplus \overline{M}}
\arrow{e,t}{x\oplus y}\node{M \oplus M}\arrow{e} \node{M/xM \oplus
M/yM}\arrow{e}\node{0,}
\end{diagram}
\]
where $\Delta_{\overline{M}}: \overline{M} \to \overline{M} \oplus
\overline{M}$, $\Delta(m) = (m,m)$ is  a diagonal homomorphism. Note  that the derived homomorphism of
$\Delta_{\overline{M}}$ is also a diagonal homomorphism, the homomorphism
$\Delta_{H^{i}_\mathfrak{a}(\overline{M})}:
H^{i}_\mathfrak{a}(\overline{M}) \to
H^{i}_\mathfrak{a}(\overline{M}) \oplus
H^{i}_\mathfrak{a}(\overline{M})$ is diagonal for all $i\geq 0$. Therefore,
we get by applying the functor $H^{i}_\mathfrak{a}(\bullet)$ to the above
diagram  the following commutative diagram
\[
\begin{diagram}
\node{\cdots}\arrow{e}\node{H^{i}_\mathfrak{a}(\overline{M})}
\arrow{e,t}{\varphi^i}\arrow{s,l}{\Delta_{H^{i}_\mathfrak{a}(\overline{M})}}\node{H^{i}_\mathfrak{a}(M)^2}\arrow{e}
\arrow{s,l}{id}\node{\cdots}\\
\node{\cdots}\arrow{e}\node{H^{i}_\mathfrak{a}(\overline{M})^2}
\arrow{e,t}{x\oplus y}\node{H^{i}_\mathfrak{a}(M)^2}\arrow{e}
\node{\cdots,}
\end{diagram}
\]
where  $A^2= A \oplus A$ for an $R$-module $A$, and $\varphi^i$ is
derived from $\overline{\varphi}$. Since $x, y$ satisfy the
condition $(\sharp)$, the homomorphism in the bottom row is zero,
for all $i<t$, hence $\varphi^i = 0$ for all $i<t$. Therefore, for
all $i<t-1$, the following diagram is commutative
\[
 \begin{CD}
{0 \longrightarrow H^{i}_\mathfrak{a}(M)^2} @>>>
{H^{i}_\mathfrak{a}(N)} @>>>
{H^{i+1}_\mathfrak{a}(\overline{M})} @>>> 0\\
   id @VVV  @VVV     @V\Delta_{H^{i+1}_\mathfrak{a}}(\overline{M})VV\\
{0 \longrightarrow H^{i}_\mathfrak{a}(M)^2} @>>>
{H^{i}_\mathfrak{a}(M/xM) \oplus H^{i}_\mathfrak{a}(M/yM)} @>>>
{H^{i+1}_\mathfrak{a}(\overline{M})^2} @>>> 0.
 \end{CD}
 \]
For all $i<t-1$, the exact sequence in the bottom row is just $E^i_x \oplus
E^i_y$. We denote the exact sequence in the top row by $E^i$, so
$$E^i =(E_x^i \oplus
E^i_y){\Delta_{H^{i+1}_\mathfrak{m}(\overline{M})}}\quad (1)$$
 for all $i<t-1$.\\
Moreover, if $H^{t}_\mathfrak{a}(M) \cong
H^{t}_\mathfrak{a}(\overline{M})$, we have the following commutative diagram
\[
 \begin{CD}
{0 \longrightarrow H^{t-1}_\mathfrak{a}(M)^2} @>>>
{H^{t-1}_\mathfrak{a}(N)} @>>>
{K_{(x,y)}} @>>> 0\\
   id @VVV  @VVV     @V\overline{\Delta}VV\\
{0 \longrightarrow H^{t-1}_\mathfrak{a}(M)^2} @>>>
{H^{t-1}_\mathfrak{a}(M/xM) \oplus H^{t-1}_\mathfrak{a}(M/yM)}
@>>> {K_x \oplus K_y} @>>> 0,
 \end{CD}
 \]
 where  $K_{(x,y)} =
0:_{H^{t}_\mathfrak{a}(\overline{M})}(x,y), K_{x} =
0:_{H^{t}_\mathfrak{a}(\overline{M})}x, K_{y} =
0:_{H^{t}_\mathfrak{a}(\overline{M})}y$, and
$\overline{\Delta}: K_{(x,y)} \longrightarrow K_{x} \oplus
K_{y}$ defined by $\quad \overline{\Delta}(c) = (c, c).$
 Since
$$\mathrm{Hom}_R(R/\mathfrak{a}^{n_0}, K_x)\cong
\mathrm{Hom}_R(R/\mathfrak{a}^{n_0}, K_y) \cong
\mathrm{Hom}_R(R/\mathfrak{a}^{n_0}, K_{(x,y)}) \cong
0:_{H^{t}_\mathfrak{a}(\overline{M})}\mathfrak{a}^{n_0},$$
by applying the functor $\mathrm{Ext}_R^i(R/\mathfrak{a}^{n_0}, \bullet)$ to
 the above diagram we
obtain the following commutative diagram
\[
\begin{CD}
{0:_{H^{t}_\mathfrak{a}(\overline{M})}\mathfrak{a}^{n_0}}
@>\delta_1>>
\mathrm{Ext}_R^1(R/\mathfrak{a}^{n_0},H^{t-1}_\mathfrak{a}(M)^2)\\
      \Delta @VVV @VidVV \\
{ (0:_{H^{t}_\mathfrak{a}(\overline{M})}\mathfrak{a}^{n_0})^2}
@>\delta_2>>
\mathrm{Ext}_R^1(R/\mathfrak{a}^{n_0},H^{t-1}_\mathfrak{a}(M)^2),
 \end{CD}
 \]
where $\delta_1, \delta_2$ are connected homomorphisms. Because
$F^{t-1}_{n_0,x}, F^{t-1}_{n_0, y}$ are determined, $\delta_2 =0$, so $\delta_1=0$.  Hence we obtain the following commutative diagram with exact rows
\[
 \begin{CD}
{0 \longrightarrow
0:_{H^{t-1}_\mathfrak{a}(M)^2}\mathfrak{a}^{n_0}} @>>>
{0:_{H^{t-1}_\mathfrak{a}(N)}\mathfrak{a}^{n_0}} @>>>
{0:_{H^{t}_\mathfrak{a}(\overline{M})}\mathfrak{a}^{n_0}} \longrightarrow 0\\
   id @VVV  @VVV      @V{\Delta}VV \\
{0 \longrightarrow
0:_{H^{t-1}_\mathfrak{a}(M)^2}\mathfrak{a}^{n_0}} @>>>
{0:_{H^{t-1}_\mathfrak{a}(M/xM) \oplus
H^{t-1}_\mathfrak{a}(M/yM)}\mathfrak{a}^{n_0}} @>>>
{0:_{H^{t}_\mathfrak{a}(\overline{M})^2}\mathfrak{a}^{n_0}}
\longrightarrow 0.
 \end{CD}
 \]
 The sequence in the bottom row is  just $F^{t-1}_{n_0, x} \oplus F^{t-1}_{n_0,y}$. We denote
the sequence in the top row by $F^{t-1}_{n_0}$, so
 $$F^{t-1}_{n_0} = (F^{t-1}_{n_0,x}\oplus
 F^{t-1}_{n_0, y}) \Delta_{0:_{H^{t}_\mathfrak{t}(\overline{M})}\mathfrak{a}^{n_0}}. \quad \quad(2)$$
On the other hand, we consider the following commutative diagram
\[\divide\dgARROWLENGTH by 2
\begin{diagram}
\node{0}\arrow{e}\node{\overline{M}}
\arrow{e,t}{\overline{\varphi}}\arrow{s,l}{id}\node{M \oplus
M}\arrow{e}\arrow{s,l}{\nabla_M}
\node{N}\arrow{e}\arrow{s} \node{0}\\
\node{0}\arrow{e}\node{\overline{M}} \arrow{e,t}{x+y}\node{M
}\arrow{e} \node{{M}/{(x+y)M}}\arrow{e}\node{0,}
\end{diagram}
\]
where $\nabla_{M}: M \oplus M \to M$, $\nabla(m, m') = m + m'$ is the
codiagonal homomorphism. Since
 derived homomorphisms of $\nabla_M$  are also  codiagonal
homomorphisms, so is the homomorphism $\nabla_{H^{i}_\mathfrak{a}(M)}:
H^{i}_\mathfrak{a}(M) \oplus H^{i}_\mathfrak{a}(M) \to
H^{i}_\mathfrak{a}(M)$ for all $i\geq 0$. Hence by applying the functor
$H^{i}_\mathfrak{a}(\bullet)$ to the above diagram we get the
following commutative diagram
\[\divide\dgARROWLENGTH by 2
\begin{diagram}
\node{E^i:\,0}\arrow{e}\node{H^{i}_\mathfrak{a}(M)\oplus
H^{i}_\mathfrak{a}(M)}\arrow{e}\arrow{s,l}{\nabla_{H^{i}_\mathfrak{a}(M)}}\node{H^{i}_\mathfrak{a}(N)}\arrow{e}\arrow{s}
\node{H^{i+1}_\mathfrak{a}(\overline{M})}\arrow{s,r}{id}\arrow{e} \node{0}\\
\node{E^i_{x+y}:\,0}\arrow{e}\node{H^{i}_\mathfrak{a}(M)}
\arrow{e}\node{H^{i}_\mathfrak{a}\big({M}/{(x+y)M}\big)}\arrow{e}
\node{H^{i+1}_\mathfrak{m}(\overline{M})}\arrow{e}\node{0,}
\end{diagram}
\]
for all $i<t-1$.  It follows  for all $i<t-1$ that
$$E^i_{x+y} = \nabla_{H^{i}_\mathfrak{m}(M)}E^i .\quad \quad (3)$$
Moreover, if $H^{t}_\mathfrak{a}(M) \cong
H^{t}_\mathfrak{a}(\overline{M})$, we have
\[
 \begin{CD}
{0 \longrightarrow H^{t-1}_\mathfrak{a}(M)^2} @>>>
{H^{t-1}_\mathfrak{a}(N)} @>>>
{K_{(x,y)}} @>>> 0\\
   @V\nabla_{H^{t-1}_\mathfrak{a}(M)}VV  @VVV     @V\mu VV\\
{0 \longrightarrow H^{t-1}_\mathfrak{a}(M)} @>>>
{H^{t-1}_\mathfrak{a}({M}/{(x+y)M})} @>>> {K_{x+y}} @>>> 0,
 \end{CD}
 \]
 where $\mu$ is injective. By applying the functor $\mathrm{Hom}_R(R/\mathfrak{a}^{n_0}, \bullet)$ to the above
 diagram we get
\[
 \begin{CD}
{0 \longrightarrow
(0:_{H^{t-1}_\mathfrak{a}(M)}\mathfrak{a}^{n_0})^2} @>>>
{0:_{H^{t-1}_\mathfrak{a}(N)}\mathfrak{a}^{n_0}} @>>>
{0:_{H^{t}_\mathfrak{a}(\overline{M})}\mathfrak{a}^{n_0}} @>>> 0\\
   @V\nabla_{0:_{H^{t-1}_\mathfrak{a}(M)}\mathfrak{a}^{n_0}}VV  @VVV    @VidVV \\
{0 \longrightarrow 0:_{H^{t-1}_\mathfrak{a}(M)}\mathfrak{a}^{n_0}}
@>>> {0:_{H^{t-1}_\mathfrak{m}(M/(x+y)M)}\mathfrak{a}^{n_0}} @>>>
{0:_{H^{t}_\mathfrak{a}(\overline{M})}\mathfrak{a}^{n_0}.}
 \end{CD}
 \]
 It follows form the existence of $F^{t-1}_{n_0}$ that the bottom row is exact, and hence $F^{t-1}_{n_0,x+y}$ is
determined. Therefore
$$F^{t-1}_{n_0,x+y} = \nabla_{0:_{H^{t-1}_\mathfrak{a}(M)}\mathfrak{a}^{n_0}}F^{t-1}_{n_0}. \quad \quad (4)$$
Combining (1) and (3), we have
$$E^i_{x+y} = \nabla_{H^{i}_\mathfrak{a}(M)}(E_x^i \oplus E^i_y)\Delta_{H^{i+1}_\mathfrak{a}(\overline{M})},$$
 for all $i<t-1$. So $E^i_{x+y} = E_x^i + E^i_y$ for all
 $i<t-1$.\\
Combining (2) and (4), we have
$$F^{t-1}_{n_0,x+y} = \nabla_{0:_{H^{t-1}_\mathfrak{a}(M)}\mathfrak{a}^{n_0}}(F^{t-1}_{n_0,x}\oplus
 F^{t-1}_{n_0, y}) \Delta_{0:_{H^{t}_\mathfrak{a}(\overline{M})}\mathfrak{a}^{n_0}}.$$
 So $F^{t-1}_{n_0,x+y} = F^{t-1}_{n_0,x}+ F^{t-1}_{n_0,y}$.\\
  (ii) The first and second claims of (ii) can be shown by the same method as used in the proof of Proposition \ref{PRO3.1}, and the last one follows  immediately from the  structure of $R$-module of the extension group Ext$^1$.
\end{proof}

\section{The proof of the main result}
First of all, we need some auxiliary lemmas.
\begin{lemma} \label{LE2.1}
Let $(R, \mathfrak{m})$ be a Noetherian local ring,
$\mathfrak{a}$, $\mathfrak{b}$  ideals  and $\mathfrak{p}_1, ...,
\mathfrak{p}_n$ prime ideals such that $\mathfrak{ab} \nsubseteq
\mathfrak{p}_j$ for all $j \leq n$. Let $x$ be an element
contained in $\mathfrak{ab}$ but $x \notin \mathfrak{p}_j$ for all
$j \leq n$. Then there are elements  $a_1, ..., a_r \in
\mathfrak{a}, b_1, ..., b_r \in \mathfrak{b}$ that  we can write
$x=a_1b_1+...+ a_rb_r$ such that $a_ib_i \notin \mathfrak{p}_j$
and $a_1b_1+...+a_ib_i \notin \mathfrak{p}_j$ for all $i \leq r, j
\leq n$.
\end{lemma}
\begin{proof} It is sufficient to prove the assertion in the
case $\mathfrak{p}_i \nsubseteq \mathfrak{p}_j$ for all $i,j \leq
n, i \neq j$. By the Prime Avoidance Theorem, we can choose  a
system of generators  $a_1b_1, ..., a_rb_r$ of
$\mathfrak{ab}$ such that  $a_i \in \mathfrak{a}, b_i \in
\mathfrak{b}$ for all $i\leq r$, and $a_ib_i \notin
\mathfrak{p}_j$ for all $i \leq r, j\leq n$. Hence there exist $s_i
\in R, \, i=1, ...,r$ such that $x=s_1a_1b_1+...+ s_ra_rb_r$.
Rewrite $x=a_1(s_1b_1)+...+ a_r(s_rb_r)$, therefore we may assume without loss of generality
that $x$ can be written in form $x=a_1b_1+a_2b_2+ ...+a_rb_r$
with $a_i \in \mathfrak{a}, b_i \in \mathfrak{b}$ for all $i\leq
r$, and $a_i \notin
\mathfrak{p}_j$ for all $i \leq r, j\leq n$.\\
We prove the assertion by induction on $r$. The case $r=1$
 is trivial.  Assume that $r>1$ and the
 lemma is true for $r-1$. Set $J = \{j\,:\,b_r \in
 \mathfrak{p}_j\}.$
 By the Prime Avoidance Theorem we can choose $u \in
 \mathfrak{b}$ such that $u \notin \mathfrak{p}_j$ for all $j \in
 J$, and $u \in \mathfrak{p}_j$ for all $j \notin
 J$. Since  $a_1\notin
\mathfrak{p}_j$ for all $j\leq n$,   $ua_1$  also has
this property. Therefore $b_r+ua_1\notin \mathfrak{p}_j$
for all $j\leq n$.
 We write $x=a_1(b_1 - ua_r)+ a_2b_2 + ...
 +a_r(b_r + ua_1)$, so without loss of generality we can assume more that $x=a_1b_1+a_2b_2+ ...+a_rb_r$ and $a_rb_r\notin
\mathfrak{p}_j$ for all $j\leq n$.
 Let $x'= a_1b_1+ ...+a_{r-1}b_{r-1}$,
and set $J' = \{j\,:\, x' \in
 \mathfrak{p}_j\}.$
Using the Prime Avoidance Theorem  again we can choose $v \in
 \mathfrak{m}$ such that $v \notin \mathfrak{p}_j$ for all $j \in
 J'$, and $v \in \mathfrak{p}_j$ for all $j \notin J'$.
 Because $a_1, a_r, b_r\notin \mathfrak{p}_j$ for all $j\leq n$,  $va_1a_rb_r$ has the same
 property as $v$.  Set $x_{r-1}=x'+va_1a_rb_r = a_1(b_1 + va_rb_r)
+a_2b_2 + ... +
 a_{r-1}b_{r-1}$
Then  $x_{r-1}\notin \mathfrak{p}_j$ for all $j\leq n$ and
$x =x_{r-1}+  a_rb_r(1 - va_1)$. Since $a_rb_r(1 - va_1)\notin \mathfrak{p}_j$ for all $j\leq n$, the conclusion follows  from the inductive hypothesis for the element $x_{r-1}$.
\end{proof}

\begin{corollary}\label{CO2.2}
Let $(R, \mathfrak{m})$ be a Noetherian local ring and $\mathfrak{a}$
 an ideal of $R$. Let $M$ be a finitely generated
$R$-module and $x \in \mathfrak{a}^2$  an $\mathfrak{a}$-filter
regular element of $M$. Then we can find   $\mathfrak{a}$-filter regular  elements  $a_1, ..., a_r, b_1,
...,b_r \in \mathfrak{a}$ of $M$ such that $x = a_1b_1 + ... +
 a_{r}b_{r}$ and $a_1b_1 + ... +
 a_{i}b_{i}$ are also $\mathfrak{a}$-filter regular elements of $M$ for all $i \leq r$.
\end{corollary}
\begin{proof}
It follows from Lemma \ref{LE2.1} with  $\frak a = \frak b$ and $\mathrm{Ass}(M) \setminus
V(\mathfrak{a})$ for the set of prime ideals $\{\mathfrak{p}_1,
..., \mathfrak{p}_n\}$.
\end{proof}
The following result is somehow known. The proof of this lemma follows easily from the commutativity of localizations  and
the functor Hom of finitely generated modules, therefore we omit it.
\begin{lemma}\label{CO3.3} Let $A, B, C$ are finitely generated $R$-modules.
Then the sequence
$$0 \longrightarrow A \longrightarrow B \longrightarrow C \longrightarrow 0$$
is a split exact sequence if and only if the sequence
 $$0 \longrightarrow A_{\frak m} \longrightarrow B_{\frak m} \longrightarrow C_{\frak m} \longrightarrow 0$$
 is  exact and split  for all maximal ideal $\frak m$ of $R$.
\end{lemma}

\noindent {\it Proof of Theorem \ref{th}}. Keep all notations as
in Section 3 with $U = H^{0}_\mathfrak{a}(M) $. Then every
$\mathfrak{a}$-filter regular element $x \in \mathfrak{a}^{n_0}$
satisfies  the condition $(\sharp)$. Since $H^{i}_\mathfrak{a}(M)
\cong H^{i}_\mathfrak{a}(\overline{M})$ for all $i>0$,
$H^{i}_\mathfrak{a}(\overline{M})$ is finitely generated, so is
$0:_{H^{t}_\mathfrak{a}(\overline{M})} \frak{a}^{n_0}$ by Theorem
1.2, \cite{AKS}.  Using localizations at maximal ideals  we may
assume by Lemma \ref{CO3.3}  that $(R,\mathfrak{m})$ is a
Noetherian local ring. Let $x\in \frak a^{2n_0}$. There are by
Corollary \ref{CO2.2} $\mathfrak{a}$-filter regular elements $a_i,
b_i \in \mathfrak{a}^{n_0}$, $i \leq r$
 such that $x=a_1b_1+...+ a_rb_r$ and  $a_1b_1+...+a_jb_j$ are $\frak a$-filter
regular elements for all $1 \leq j \leq r$. Then, by virtue of Theorem
\ref{T3.2} (i) we have
$$E^i_x   = E^i_{a_1b_1+...+ a_rb_r}
        =E^i_{a_1b_1} + E^i_{a_2b_2} + \cdots + E^i_{a_rb_r}.$$
 Therefore $$E^i_x = a_1E^i_{b_1} + a_2E^i_{b_2} + \cdots + a_rE^i_{b_r} = 0$$  by Theorem
\ref{T3.2} (ii) for all  $0\leq i<t-1$.  Thus  we have
$$H^{i}_\mathfrak{a}(M/xM)\cong H^{i}_\mathfrak{a}(M)\oplus H^{i+1}_\mathfrak{a}(\overline M)\cong H^{i}_\mathfrak{a}(M)\oplus H^{i+1}_\mathfrak{a}(M)$$
for all $0\leq i <t-1$. On the other hand,
 by Proposition \ref{PRO3.1}
$F^{t-1}_{n_0,a_jb_j}$ are determined for all $j\leq r$. It
follows by Theorem \ref{T3.2} (i) that
 $F^{t-1}_{n_0,x}=F^{t-1}_{n_0,a_1b_1+ \cdots +a_rb_r}$ is determined and
$$F^{t-1}_{n_0,x} = F^{t-1}_{n_0,a_1b_1}+ \cdots +F^{t-1}_{n_0,a_rb_r}.$$
Therefore  $F^{t-1}_{n_0,x} = 0$ by Theorem \ref{T3.2} (ii), so
$$0:_{H^{t-1}_\mathfrak{a}({M}/{xM})}\mathfrak{a}^{n_0} \cong
H^{t-1}_\mathfrak{a}(M)\oplus
0:_{H^{t}_\mathfrak{a}(M)}\mathfrak{a}^{n_0},$$ since $0:_{H^{t-1}_\mathfrak{a}(M)} \mathfrak{a}^{n_0} =
H^{t-1}_\mathfrak{a}(M)$.

\hspace*{14.5cm}$\Box$

\section{Some applications}
The first immediate consequence of Theorem \ref{th} is an
affirmative complete  answer for the question posed in the
introduction.
\begin{corollary}\label{T1.1}
Let $M$ be a generalized Cohen-Macaulay module over a local ring $(R,\frak m)$ of dimension $d>0$,
and $n_0$ the least positive integer such that
$\mathfrak{m}^{n_0}H^{i}_\mathfrak{m}(M) = 0$ for all $i <d$.  Then
for any parameter element $x \in \mathfrak{m}^{2n_0}$, we have
$$H^{i}_\mathfrak{m}(M/xM)\cong H^{i}_\mathfrak{m}(M)\oplus H^{i+1}_\mathfrak{m}(M),$$
 for all $i<d-1$, and
$$0:_{H^{d-1}_\mathfrak{m}(M/xM)}\mathfrak{m}^{n_0} \cong
H^{d-1}_\mathfrak{m}(M)\oplus
0:_{H^{d}_\mathfrak{m}(M)}\mathfrak{m}^{n_0}.$$
\end{corollary}
The next application of Theorem \ref {th} is somehow strange to the authors and it can be used to derive many consequences.
\begin{corollary}\label{LE4.1} Let $M$ be a finitely generated module over a Noetherian ring $R$ and  $\mathfrak{a}$ an
ideal of  $R$. Let $t, n_0$ be positive integers such that
$\mathfrak{a}^{n_0}H^{i}_\mathfrak{a}(M) = 0$ for all $i<t$. Let
$x_1, ..., x_t$ be an $\frak a$-filter regular sequence of $M$
contained in $\mathfrak{a}^{2n_0}$. Then for all positive integer
$k \leq n_0$ and all $j = 1, \ldots , t$, $\mathrm{Hom}_R(R/{\frak
a}^k, M/(x_1,...,x_j)M) $ are independent of the choice of the
sequence $x_1, ..., x_j$. Moreover,  we have
$$\mathrm{Hom}_R(R/{\frak a}^k, M/(x_1,...,x_j)M) \cong \bigoplus_{i=0}^j
\mathrm{Hom}_R(R/{\frak a}^k,
H^{i}_\mathfrak{a}(M))^{\binom{j}{i}}.$$
\end{corollary}
\begin{proof}
We proceed  by induction on $j$. From Theorem \ref{th} we have
$$
\mathrm{Hom}_R(R/{\frak a}^k, H_{\frak a}^i(M/x_1M)) \cong
\mathrm{Hom}_R(R/{\frak a}^k, H^{i}_\mathfrak{a}(M)) \oplus
\mathrm{Hom}_R(R/{\frak a}^k, H^{i+1}_\mathfrak{a}(M))$$ for all
$i\leq t-1$. Therefore
\begin{eqnarray*}
\mathrm{Hom}_R(R/{\frak a}^k, M/(x_1)M) &\cong&
\mathrm{Hom}_R(R/{\frak a}^k, H_{\frak a}^0(M/x_1M)) \\
&\cong& \mathrm{Hom}_R(R/{\frak a}^k, H^{0}_\mathfrak{a}(M)) \oplus
\mathrm{Hom}_R(R/{\frak a}^k, H^{1}_\mathfrak{a}(M)),
\end{eqnarray*}
and the corollary is proved for $j=1$. Suppose that $j>1$. By Theorem \ref{th} we have
$\mathfrak{a}^{n_0}H^{i}_\mathfrak{a}(M/x_1M) = 0$ for all
$i<t-1$. It follows from the inductive hypothesis for the sequence $x_2, ..., x_j$ and $M/x_1M$ that
\begin{eqnarray*}
\mathrm{Hom}_R(R/{\frak a}^k, M/(x_1,...,x_j)M) &\cong&
\bigoplus_{i=0}^{j-1} \mathrm{Hom}_R(R/{\frak a}^k,
H^{i}_\mathfrak{a}(M/x_1M))^{\binom{j-1}{i}}\\
&\cong& \bigoplus_{i=0}^j\mathrm{Hom}_R(R/{\frak a}^k,
H^{i}_\mathfrak{a}(M))^{\binom{j}{i}}
\end{eqnarray*}
as required.
\end{proof}
Let   $M$ be a finitely generated module over a Noetherian local
ring $(R, \frak m)$ and $\mathfrak{q}$ a parameter ideal of $M$.
The  index of irreducibility of  $\mathfrak{q}$ on $M$ is  defined
by $N_R(\mathfrak{q},M) =
\dim_{R/\mathfrak{m}}\mathrm{Soc}(M/\mathfrak{q}M)$, where
$\mathrm{Soc}(N) \cong 0:_N \mathfrak{m} \cong
\mathrm{Hom}(R/\mathfrak{m}, N)$ for an arbitrary $R$-module $N$.
It is well-known that if $M$ is a Cohen-Macaulay module then
$N_R(\mathfrak{q},M)$ is a constant independent of the choice of
$\mathfrak{q}$. In the case $M$ is a Buchsbaum module, S.  Goto
and H. Sakurai proved in \cite{GS}  that for large enough $n$ the
index of irreducibility $N_R(\mathfrak{q},M)$ is a constant for
all parameter ideals  $\mathfrak{q}$ contained in $\mathfrak{m}^n
$. And they conjectured that this result is also true for
generalized Cohen-Macaulay modules.   H.L. Truong and the first
author  have given an affirmative answer for this conjecture in
\cite{CT}. Now, in virtue of Corollary  \ref {LE4.1} we can prove
a statement which is a slight generalization of the main result of
\cite {CT} as follows.
\begin{corollary} \label{CO1.2}
Let $M$ be a generalized Cohen-Macaulay module of dimension $d$
and $n_0$  a positive integer such that
$\mathfrak{m}^{n_0}H^{i}_\mathfrak{m}(M) = 0$ for all $i<d$. Then,
for every parameter ideal  $\mathfrak{q}$ of $M$ contained in
$\mathfrak{m}^{2n_0}$ and $k \leq n_0$,  the length
$\ell_R\big(({\mathfrak{q}M:_M
\mathfrak{m}^k})/{\mathfrak{q}M}\big)$  is independent of the
choice of $\frak q$ and given by
$$\ell_R\big(({\mathfrak{q}M:_M
\mathfrak{m}^k})/{\mathfrak{q}M}\big) =
\sum_{i=0}^{d}\binom{d}{i}\ell_R(0:_{H^{i}_\mathfrak{m}(M)}\mathfrak{m}^k)
.$$ In
particular, the index of irreducibility  $N_R(\mathfrak{q},M)$ is a constant
independent of the choice of $\mathfrak{q}$ and
$$N_R(\mathfrak{q},M)=
\sum_{i=0}^{d}\binom{d}{i}\dim_{R/\mathfrak{m}}
\mathrm{Soc}({H^{i}_\mathfrak{m}(M)}).$$
\end{corollary}
\begin{proof}
It follows immediately from Corollary \ref{LE4.1} and the fact
that \\
$\mathrm{Hom}_R(R/{\frak m}^k, M/\frak qM) \cong
({\mathfrak{q}M:_M \mathfrak{m}^k})/{\mathfrak{q}M} $  and
$\mathrm{Hom}_R(R/{\frak m}^k, H^{i}_\mathfrak{m}(M))\cong
0:_{H^{i}_\mathfrak{m}(M)}\mathfrak{m}^k$
 for all $i$.
\end{proof}
  In \cite{H} C. Huneke conjectured that  the set $\mathrm{Ass}\, H^i_\mathfrak{a}(M)$  is a finite set for  any ideal $\frak a$ and all $i$.
 The conjecture was settled by G. Lyubeznik \cite{L} and  C. Huneke - R.Y. Sharp \cite{HS} for  regular local rings containing a field.
Although M.  Katzman \cite{K} has given an example of a Noetherian
ring and an ideal $\mathfrak{a}$ such that $H^2_\mathfrak{a}(R)$
has infinitely many associated primes, the conjecture is still
true in many interesting cases.  The following result is an
immediate consequence of Corollary \ref{LE4.1}, which is an
extension of the  main result of  M. Brodmann and A.L. Faghani in
\cite{BF}.
\begin{corollary}
Let  $M$ be a finitely generated $R$-module and $\mathfrak{a}$  an
ideal of $R$. Let $t$ and $n_0$ be positive integers such that
$\mathfrak{a}^{n_0}H^{i}_\mathfrak{a}(M) = 0$ for all $i<t$.  Then
for every $\mathfrak{a}$-filter regular sequence $x_1, ..., x_t$
of $M$ contained in $\mathfrak{a}^{2n_0}$, we have
$$\bigcup_{i=0}^j\mathrm{Ass}\, H^i_\mathfrak{a}(M) = \mathrm{Ass}\,(M/(x_1,...,x_jM)) \bigcap V(\mathfrak{a})$$
for all $j= 1, \ldots , t$.
In particular, $H^t_\mathfrak{a}(M)$ has only finitely many
associated primes.
\end{corollary}
\begin{proof}  Since $H^i_\mathfrak{a}(M)$ is
$\frak a$-torsion,  $\mathrm{Ass}\, H^i_\mathfrak{a}(M) =
\mathrm{Ass}\, \mathrm{Hom}_R(R/{\frak a},
H^{i}_\mathfrak{a}(M))$.  It follows from Corollary \ref{LE4.1} that for all $j= 1, \ldots , t$,
$$\bigcup_{i=0}^j \mathrm{Ass}\, H^i_\mathfrak{a}(M) = \mathrm{Ass}\, H^0_\mathfrak{a}(M/(x_1,...,x_j)M )=\mathrm{Ass}\,(M/(x_1,...,x_j)M) \cap V(\mathfrak{a}).$$
\end{proof}

\textsc{Institute of Mathematics, 18 Hoang Quoc Viet Road, 10307
Hanoi, Viet Nam}\\
{\it E-mail address}: ntcuong@math.ac.vn\\
\textsc{Department of Mathematics, FPT University (Dai Hoc FPT), 8 Ton That Thuyet Road, Ha Noi, Viet Nam}\\
 {\it E-mail address}: phamhungquy@gmail.com {\it or} quyph@fpt.edu.vn

\begin{thebibliography}{99}
\bibitem{AKS} J. Asadollahi and K.
Khashyarmanesh and Sh. Salarian, "On the finiteness properties of
generalized local cohomology modules", {\it Comm. Algebra},
30(2002), 859-867.
\bibitem{BF} M. Brodmann and A. L. Faghani, A finiteness result for associated primes of local cohomology modules,
{\it Proc. Amer. Math. Soc.}, 128(2000), 2851-2853.
\bibitem{CST} N. T. Cuong and P. Schenzel and N. V. Trung,
Verallgeminerte Cohen-Macaulay moduln, {\it Math-Nachr.},
85(1978), 156-177.
\bibitem{CT} N. T. Cuong and H. L. Truong, Asymptotic behavior of parameter ideals in generalized Cohen-Macaulay
module, {\it J. Algebra}, 320(2008), 158-168.
\bibitem{GS} S. Goto and H. Sakurai, The equality $I^2 = QI$ in Buchsbaum rings, {\it Rend. Sem. Univ. Padova.},
110(2003), 25-56.
\bibitem{H} C. Huneke, {\it Problems on local cohomology}, Free resolutions in commutative algebra and algebraic geometry
(Sundance, Utah, 1990), Res. Notes Math., {\bf 2} (1992), 93-108.
\bibitem{HS} C. Huneke and R. Sharp, Bass numbers of local cohomology modules, {\it Trans.
Amer. Math. Soc.}, 339(1993), 765-779.
\bibitem{K} M. Katzman, An example of an infinite set of associated primes of a local cohomology module, {\it J.
Algebra}, 252(2002), 161-166.
\bibitem{L} G. Lyubeznik, Finiteness properties of local cohomology modules (an application of $D$-modules to
commutative algebra), {\it Invent. Math}, 113(1993), 41-55.
\bibitem{M} S. MacLane, {\it Homology}, Springer-Verlag, third edition,
1975.
\bibitem{S} N. Suzuki, On quasi-Buchsbaum modules. An application of theory of FLC-modules.
{\it Commutative algebra and combinatorics (Kyoto, 1985)}, 215-243, Adv.
Stud. Pure Math., 11, North-Holland, Amsterdam, 1987.
\end{thebibliography}
\end{document}